\documentclass[12pt]{amsart}
\usepackage{amscd,amsfonts,amssymb,amsmath,amsthm,latexsym,mathrsfs,graphicx}
\usepackage[margin=1in]{geometry}
\usepackage[bookmarks, colorlinks=true, linkcolor=blue, citecolor=blue, urlcolor=blue]{hyperref}

\newcommand{\bbA}{\mathbb{A}}
\newcommand{\bbB}{\mathbb{B}}
\newcommand{\bbC}{\mathbb{C}}
\newcommand{\bbD}{\mathbb{D}}

\newcommand{\bbQ}{\mathbb{Q}}

\newcommand{\bbR}{\mathbb{R}}

\newcommand{\bbZ}{\mathbb{Z}}

\newcommand{\bbZp}{{\mathbb{Z}_p}}

\newcommand{\bfL}{\mathbf{L}}

\newcommand{\bfR}{\mathbf{R}}

\newcommand{\bs}{\backslash}

\newcommand{\calG}{\mathcal{G}}
\newcommand{\calH}{\mathcal{H}}

\newcommand{\calL}{\mathcal{L}}

\newcommand{\calO}{\mathcal{O}}
\newcommand{\calR}{\mathcal{R}}

\DeclareMathOperator{\chr}{char}

\newcommand{\cn}{\colon}

\DeclareMathOperator{\coker}{coker}

\newcommand{\cont}{\mathrm{cont}}

\newcommand{\de}{\delta}

\newcommand{\dR}{\mathrm{dR}}

\newcommand{\ep}{\epsilon}

\DeclareMathOperator{\Exp}{Exp}

\DeclareMathOperator{\Fib}{Fib}
\DeclareMathOperator{\Fil}{Fil}

\newcommand{\fkS}{\mathfrak{S}}
\newcommand{\fks}{\mathfrak{s}}
\newcommand{\fkT}{\mathfrak{T}}
\newcommand{\fkt}{\mathfrak{t}}

\newcommand{\Frob}{\mathrm{Frob}}

\newcommand{\Ga}{\Gamma}
\newcommand{\ga}{\gamma}
\DeclareMathOperator{\Gal}{Gal}

\DeclareMathOperator{\GL}{GL}

\DeclareMathOperator{\Hom}{Hom}

\newcommand{\Iw}{\mathrm{Iw}}
\DeclareMathOperator{\Ind}{Ind}
\newcommand{\inv}{^{-1}}

\newcommand{\La}{\Lambda}
\newcommand{\la}{\lambda}

\DeclareMathOperator{\loc}{loc}

\newcommand{\Lotimes}{\mathop{\stackrel{\bfL}{\otimes}}\displaylimits}
\newcommand{\llim}{\varprojlim}

\newcommand{\Nekovar}{{Nekov\'a\v{r}}}

\newcommand{\ov}[1]{{\overline{#1}}}

\DeclareMathOperator{\Pic}{Pic}

\DeclareMathOperator{\pr}{pr}
\DeclareMathOperator{\proj}{proj}

\DeclareMathOperator{\rank}{rank}

\DeclareMathOperator{\res}{res}
\newcommand{\rig}{\mathrm{rig}}

\newcommand{\scrC}{\mathscr{C}}

\newcommand{\scrE}{\mathscr{E}}

\newcommand{\scrG}{\mathscr{G}}

\newcommand{\Sel}{\mathrm{Sel}}

\DeclareMathOperator{\Sym}{Sym}

\DeclareMathOperator{\Tor}{Tor}

\newcommand{\tors}{\mathrm{tors}}

\DeclareMathOperator{\Tw}{Tw}

\newcommand{\vphi}{\varphi}

\newcommand{\wt}[1]{{\widetilde{#1}}}

\newcommand{\BC}{\mathrm{BC}}
\DeclareMathOperator{\Col}{Col}
\DeclareMathOperator{\cond}{cond}
\newcommand{\cris}{\mathrm{cris}}
\DeclareMathOperator{\invt}{inv}
\newcommand{\prop}{{\mathrm{pro}\text{-}p}}

\DeclareMathOperator{\summ}{sum}
\newtheorem{theorem}{Theorem}[section]

\newtheorem{lemma}[theorem]{Lemma}

\theoremstyle{definition}
\newtheorem{definition}[theorem]{Definition}

\newtheorem{remark}[theorem]{Remark}

\hypersetup{pdfauthor={Robert Harron and Jonathan Pottharst},pdftitle={Iwasawa theory for symmetric powers of CM
  modular forms at nonordinary primes, II},pdfkeywords={Iwasawa theory, Main Conjectures, CM modular forms, Symmetric powers, Finite-slope Selmer modules}}

\begin{document}

\setcounter{section}{6}
\setcounter{equation}{25}

\title[Symmetric powers, II]{Iwasawa theory for symmetric powers of CM
  modular forms at nonordinary primes, II}
\author{Robert Harron}
\author{Jonathan Pottharst}
\date\today
\subjclass[2010]{11R23, 11F80, 11F67}
\keywords{Iwasawa theory, Main Conjectures, CM modular forms, Symmetric powers, finite-slope Selmer modules}
\thanks{The first author is partially supported by NSA Young Investigator Grant \#H98230-13-1-0223 and NSF RTG Grant ``Number Theory and Algebraic Geometry at the University of Wisconsin''.}

\begin{abstract}
Continuing the study of the Iwasawa theory of symmetric powers of CM
modular forms at supersingular primes begun by the first author and
Antonio Lei, we prove a Main Conjecture equating the ``admissible''
$p$-adic $L$-functions to the characteristic ideals of
``finite-slope'' Selmer modules constructed by the second author.  As
a key ingredient, we improve Rubin's result on the Main Conjecture of
Iwasawa theory for imaginary quadratic fields to an equality at inert
primes.
\end{abstract}
\maketitle
\tableofcontents

\section*{Introduction}

The study of the Iwasawa theory of symmetric powers of CM modular
forms at supersingular primes was begun by the first author and
Antonio Lei in \cite{HL}.  They constructed two types of $p$-adic
$L$-functions: ``admissible'' ones in the sense of Panchishkin and
Dabrowski, and ``plus and minus'' ones in the sense of Pollack.  They
also constructed ``plus and minus'' Selmer modules in the sense of
Kobayashi, and, using Kato's Euler system, they compared them to the
latter $p$-adic $L$-functions via one divisibility in a main
conjecture.  The present paper performs the analogous comparison
between the admissible $p$-adic $L$-functions and the ``finite-slope''
Selmer modules in the sense of the second author.  In order to get an
identity of characteristic ideals, rather than just a divisibility, we
improve the work of Rubin on the Main Conjecture of Iwasawa theory for
imaginary quadratic fields at inert primes \cite{R,R2} to give an
equality unconditionally.  Rubin's work has since been used by various
authors to derive other divisibilities; an examination of these
derivations will show that our work upgrades most of these
divisibilities to identities.

The first section of this paper is written as a direct continuation of
\cite{HL}; all numbered references to equations, theorems, etc.\ in it
are to the two papers commonly, except for bibliographical citations,
which are to the references section here.  In this section, we recall
the relevant setup from \cite{HL}, as well as the theory of
finite-slope Selmer groups from \cite{P}.  Then we give our results
about finite-slope Selmer modules of CM modular forms and their
symmetric powers at supersingular primes.  The second section is
written independently of \cite{HL} and the first section.  In it we
recall the notations from \cite{R,R2} and then treat the Iwasawa
theory of imaginary quadratic fields at inert primes.

\subsection*{Acknowledgement}

The authors would like to thank Robert Pollack and Karl Rubin for
helpful conversations and correspondence.

\section{CM modular forms and their symmetric powers}

\subsection{Notations and hypotheses of \cite{HL}}

The prime $p$ is assumed odd.  We fix algebraic closures and
embeddings $\iota_\infty \cn \ov\bbQ \to \bbC$ and $\iota_p \cn
\ov\bbQ \to \ov\bbQ_p$, and use these for the definition of Galois
groups and decomposition groups.  In particular, we write $c \in
\Gal(\ov\bbQ/\bbQ)$ for the complex conjugation induced by
$\iota_\infty$.

We normalize reciprocity maps of class field theory to send
uniformizers to arithmetic Frobenius elements.  If $E/\bbQ_p$ is a
finite extension, we normalize duals of $E$-linear Galois
representations by $V^* = \Hom_E(V,E(1))$, and Fontaine's functors by
$\bbD_\cris(V) = \Hom_{\bbQ_p[G_{\bbQ_p}]}(V,\bbB_\cris)$ and
$\wt\bbD_\cris(V) = (\bbB_\cris \otimes_{\bbQ_p} V)^{G_{\bbQ_p}}$.

For $n \leq \infty$ we write $k_n = \bbQ(\mu_{p^n})$ and $\bbQ_{p,n} =
\bbQ_p(\mu_{p^n})$.  The cyclotomic character $\chi$ induces an
isomorphism $G_\infty := \Gal(\bbQ_{p,\infty}/\bbQ_p) \cong
\bbZ_p^\times$, and $G_\infty$ factors uniquely as $\Delta \times \Ga$
in such a way that $\chi$ induces isomorphisms $\Delta \cong
\mu_{p-1}$ and $\Ga \cong 1+p\bbZ_p$.  We fix a topological generator
$\ga_0$ of $\Ga$.

For a finite extension $E$ of $\bbQ_p$ and $G=G_\infty,\Ga$, we write
$\La_{\calO_E}(G) = \calO_E[\![G]\!]$ for the Iwasawa algebra of $G$
with coefficients in $\calO_E$ and we write $\La_E(G) =
\La_{\calO_E}(G)\otimes_{\calO_E} E$.  We let $\calH_{r,E}(G)$ be the
$E$-valued $r$-tempered distributions on $G$ for $r \in \bbR_{\geq 0}$
and, $\calH_{\infty,E}(G) = \bigcup_r \calH_{r,E}(G)$.  These objects
are stable under the involution $\iota$ (resp.\ twisting operator
$\Tw_n$ for $n \in \bbZ$) obtained by continuity and $E$-linearity by
the rule $\sigma \mapsto \sigma^{-1}$ (resp.\ $\sigma \mapsto
\chi(\sigma)^n\sigma$) on group elements $\sigma \in G$.  If $G$ acts
on $M$, then $M^\iota$ denotes $M$ with $G$ action composed through
$\iota$.

We fix an imaginary quadratic field $K \subset \ov\bbQ$, considered as
a subfield of $\bbC$ via $\iota_\infty$, with ring of integers $\calO$
and quadratic character $\ep_K \cn \Gal(K/\bbQ) \cong \{\pm1\}$.  We
assume $p$ is inert in $K$, i.e.\ $\ep_K(p)=-1$, and write $\calO_p$
(resp.\ $K_p$) for the completion of $\calO$ (resp.\ $K$) at $p\calO$.

We fix a newform $f$ of weight $k \geq 2$, level $\Ga_1(N)$ with $p
\nmid N$, character $\ep$, and CM by $K$.  We write $\psi$ and $\psi^c
= \psi \circ c$ for the algebraic Hecke characters of $K$ associated
to $f$, and order them to have types $(k-1,0)$ and $(0,k-1)$,
respectively.  We write $E$ for a finite extension of $\bbQ_p$
containing $\iota_p\iota_\infty^{-1}\psi(\bbA_{K,f}^\times)$.  Note
that $E$ contains $\iota_p(K)$ and the images of the coefficients of
$f$ under $\iota_p\iota_\infty^{-1}$.  We write $V_\psi$ for the
one-dimensional $E$-linear Galois representation attached to $\psi$,
so that when $v \nmid p\cond(\psi)$ the action of $\Frob_v$ on
$V_\psi$ is by multiplication by $\psi(v)$.  We write $V_f$ for the
$E$-linear dual of the two-dimensional Galois representation
associated to $f$ by Deligne, with structure map $\rho_f \cn G_\bbQ
\to \GL(V_f)$, satisfying $\det(\rho_f) = \ep\chi^{k-1}$.  One has
$V_f \cong \Ind_K^\bbQ V_\psi$.  Since $p$ is inert in $K$, the
comparison of $L$-factors between $f$ and $\psi$ gives
$x^2-a_p(f)x+\ep(p)p^{k-1} = x^2-\psi(p)$, and in particular
$a_p(f)=0$ so that $f$ is nonordinary at $p$.  After perhaps enlarging
$E$, we fix a root $\alpha \in E$ of this polynomial, so that the
other root is $\ov\alpha=-\alpha$, and
\[
\psi(p) = \psi^c(p) = -\ep(p)p^{k-1} = -\alpha\ov\alpha = \alpha^2 =
\ov\alpha^2.
\]

Let $m \geq 1$ be an integer, and write $r = \lfloor m/2 \rfloor$ and
$\wt r = \lceil m/2 \rceil$.  We define $V_m = \Sym^m(V_f) \otimes
\det(\rho_f)^{-r}$.  There exist newforms $f_i$ for $0 \leq i \leq \wt
r-1$ (Proposition~3.4), of respective weights $k_i = (m-2i)(k-1)+1$,
levels $\Ga_1(N_i)$ with $p \nmid N_i$, characters $\ep_i$, and having
CM by $K$ (in particular, they are nonordinary at $p$), such that
\[
V_m \cong
\bigoplus_{i=0}^{\wt r-1}
  \left( V_{f_i} \otimes \chi^{(i-r)(k-1)} \right)
\oplus
\begin{cases}
\ep_K^r & m\text{ even}, \\
0 & m\text{ odd}.
\end{cases}
\]
As a consequence, the complex $L$-function (Corollary~3.5), Hodge
structure (Lemma~3.6), critical twists (Lemma~3.7), and structure of
$\bbD_\cris$ as a filtered $\vphi$-module (Lemmas~3.9 and 3.10), for
$V_m$ are all computed explicitly.  The same computations show that
the roots of $x^2+\ep_i(p)p^{k_i-1}$ are
$\alpha_i,\ov\alpha_i=-\alpha_i$, where
\[
\alpha_i = \begin{cases}
p^{(r-i)(k-1)} & m\text{ even}, \\
\alpha p^{(r-i)(k-1)} & m\text{ odd}.
\end{cases}
\]

For $\eta$ a Dirichlet character of prime-to-$p$ conductor, we denote
by $L_\eta$ its $p$-adic $L$-function (Theorem~4.1), considered as an
element of $\La_{\calO_E}(G_\infty)$ if $\eta$ is nontrivial and of
$[(\ga_0-1)(\ga_0-\chi(\ga_0))]^{-1}\La_{\calO_E(G_\infty)}$ if $\eta$
is the trivial character ${\bf1}$.  Let $\wt L_\eta \in
\La_{\calO_E}(G_\infty)$ then denote the regularized $p$-adic
$L$-function: if $\eta = {\bf1}$ then it is defined in \S5.2 by
removing the poles of $L_{\bf1}$, and otherwise it is defined to be
$L_\eta$.  Since the roots $\alpha_i,\ov\alpha_i$ of
$x^2+\ep_i(p)p^{k_i-1}$ have $p$-adic valuation $h_i := \frac{k_i-1}2
< k_i-1$, there are $p$-adic $L$-functions
$L_{f_i,\alpha_i},L_{f_i,\ov\alpha_i} \in \calH_{h_i,E}(G_\infty)$
(Theorem~4.2).  We let $\fkT$ denote the collection of tuples $\fkt =
(\fkt_0,\ldots,\fkt_{\wt r-1})$, where each $\fkt_i \in
\{\alpha_i,\ov\alpha_i\}$.  For each $\fkt \in \fkT$, we define the
\emph{admissible $p$-adic $L$-functions}
\[
L_{V_m,\fkt} =
\iota \left(\prod_{i=0}^{\wt r-1} \Tw_{(r-i)(k-1)} L_{f_i,\fkt_i}\right)
\cdot
\begin{cases}
L_{\ep_K^r} & m\text{ even}, \\
1 & m\text{ odd},
\end{cases}
\]
as well as their regularized variants $\wt L_{V_m,\fkt}$ where
$L_{\ep_K^r}$ is replaced by $\wt L_{\ep_K^r}$.  (The twist $\iota$
and the indexing are our only changes in conventions from \cite{HL}.
There, the index set $\fkS = \{\pm\}^{\wt r-1}$ is used, and $\fks \in
\fkS$ corresponds to $\fkt \in \fkT$ where $\fkt_i = \fks_i
p^{(r-i)(k-1)}$ if $m$ is even, and $\fkt_i = \fks_i \alpha
p^{(r-i)(k-1)}$ if $m$ is odd.)  Just as in the case $m=1$, these
functions can be decomposed in terms of appropriate products of twists
of ``plus and minus'' logarithms and ``plus and minus'' $p$-adic
$L$-functions (Corollary~6.9); their trivial zeroes and
$\calL$-invariants are known (Theorem~6.13), using work of Benois.

Finally, for $\theta = {\bf 1},\ep_K$, recall the Selmer groups
$\Sel_{k_\infty}(A_\theta^*)$ of equation (8), whose Pontryagin duals
$\Sel_{k_\infty}(A_\theta^*)^\vee$ are finitely generated, torsion
$\La_{\calO_E}(G_\infty)$-modules.

\subsection{Finite-slope Selmer complexes}\label{S:selmer}

For $\calG=G_\infty,\Ga$, we write $\calH_E(\calG)$ for the $E$-valued
locally analytic distributions on $\calG$; explicitly, one has
\[
\calH_E(\Ga) =
\left\{
\sum_{n \geq 0} c_n \cdot (\ga_0-1)^n \in E[\![\ga_0-1]\!]
\cn
\lim_{n \to \infty} |c_n|s^n = 0\ \text{ for all } 0 \leq s < 1
\right\},
\]
and $\calH_E(G_\infty) = \calH_E(\Ga) \otimes_E E[\Delta]$.  This ring
contains $\calH_{\infty,E}(\calG)$, and the subalgebra $\La_E(\calG)$
(hence also $\calH_{\infty,E}(\calG)$) is dense for a Fr\'echet
topology.  Although the ring is not Noetherian, it is a product of
B\'ezout domains if $\calG=G_\infty$ (resp.\ is a B\'ezout domain if
$\calG=\Ga$) as well as a Fr\'echet--Stein algebra, so that the
coadmissible $\calH_E(\calG)$-modules (in the sense of \cite{ST}) form
an abelian full subcategory of all $\calH_E(\calG)$-modules.
Coadmissible $\calH_E(\calG)$-modules include the finitely generated
ones, and have similar properties to finitely generated modules over a
product of PIDs if $\calG = G_\infty$ (resp.\ over a PID if $\calG =
\Ga$), including a structure theory and a notion of characteristic
ideal.  The algebra map $\La_E(\calG) \to \calH_E(\calG)$ is
faithfully flat so that the operation $M \mapsto M
\otimes_{\La_E(\calG)} \calH_E(\calG)$ is exact and fully faithful.
If $M$ is a finitely generated, torsion $\La_E(\calG)$-module, then
the operation is especially simple: the natural map $M
\xrightarrow{\otimes1} M \otimes_{\La_E(\calG)} \calH_E(\calG)$ is an
isomorphism, $\chr_{\calH_E(\calG)} M = (\chr_{\La_E(\calG)}
M)\calH_E(\calG)$, and, since $\calH_E(\calG)^\times =
\La_E(\calG)^\times$, all generators of this ideal actually belong to
$\chr_{\La_E(\calG)} M$.

Write $S$ for the set of primes dividing $Np$, write $\bbQ_S$ for the
maximal extension of $\bbQ$ inside $\ov\bbQ$ unramified outside $S
\cup \{\infty\}$, and let $G_{\bbQ,S} = \Gal(\bbQ_S/\bbQ)$ denote the
corresponding quotient of $G_\bbQ$.  Recall that $k_\infty \subset
\bbQ_S$, and that the natural map from $G_\infty$ to the quotient
$\Gal(k_\infty/\bbQ)$ of $G_{\bbQ,S}$ is an isomorphism; we henceforth
identify $G_\infty$ with this quotient of $G_{\bbQ,S}$.  The embedding
$\iota_p$ determines a decomposition group $G_p \subset G_{\bbQ,S}$,
and choosing additional algebraic closures and emeddings $\iota_\ell
\cn \ov\bbQ \hookrightarrow \ov\bbQ_\ell$ similarly determines
decomposition groups $G_\ell \subset G_{\bbQ,S}$ for each $\ell \mid
N$.  If $X$ is a continuous representation of $G_{\bbQ,S}$ and $G$ is
one of $G_{\bbQ,S}$ or $G_v$ with $v \in S$, we write $\bfR\Ga(G,X)$
for the class in the derived category of the complex of continuous
cochains of $G$ with coefficients in $X$, and we write $H^*(G,X)$ for
its cohomology.

We write $\La_E(G_\infty)^\iota$ (resp. $\calH_E(G_\infty)^\iota$) for
$\La_E(G_\infty)$ (resp.\ $\calH_E(G_\infty)$) considered with
$G_\infty$-action, and hence also $G_{\bbQ,S}$-action, with $g \in
G_\infty$ acting by multiplication by $g^{-1} \in G_\infty \subset
\La_E(G_\infty)^\times \subset \calH_E(G_\infty)^\times$.  If $V$ is a
continuous $E$-linear $G_{\bbQ,S}$-representation, then its classical
Iwasawa cohomology over $G=G_{\bbQ,S},G_v$ ($v \in S$) is defined by
choosing a $G_{\bbQ,S}$-stable $\calO_E$-lattice $T \subset V$ and
forming $[\llim_n H^*(G \cap \Gal(\bbQ_S/\bbQ(\mu_{p^n})),T)]
\otimes_{\calO_E} E$; a variant of Shapiro's lemma identifies it with
$H^*(G,V \otimes_E \La_E(G_\infty)^\iota)$, and in particular it is
canonically independent of the choice of lattice $T$.  The natural map
\[
H^*(G,V \otimes_E \La_E(G_\infty)^\iota)
  \otimes_{\La_E(G_\infty)} \calH_E(G_\infty)
\to
H^*(G,V \otimes_E \calH_E(G_\infty)^\iota)
\]
is an isomorphism.  We define $\bfR\Ga_\Iw(G,V) = \bfR\Ga(G,V
\otimes_E \calH_E(G_\infty)^\iota)$ and $H^*_\Iw(G,V) = H^*(G,V
\otimes_E \calH_E(G_\infty)^\iota)$.  We refer to $H^*_\Iw(G,V)$ as
the \emph{rigid analytic Iwasawa cohomology}, or, because we have no
use for classical Iwasawa cohomology in this paper, simply the
\emph{Iwasawa cohomology}.  Iwasawa cohomology groups are coadmissible
$\calH_E(G_\infty)$-modules.

There is an equivalence of categories $V \mapsto \bbD_\rig(V)$ between
continuous $E$-linear $G_p$-representations and
$(\vphi,G_\infty)$-modules over $\calR_E = \calR \otimes_{\bbQ_p} E$,
where $\calR$ is the Robba ring.  Given any $(\vphi,G_\infty)$-module
$D$ over $\calR$, we define $\bfR\Ga_\Iw(G_p,D)$ to be the class of
\[
[D \xrightarrow{1-\psi} D]
\]
in the derived category, where $\psi$ is the canonical left inverse to
$\vphi$ and the complex is concentrated in degrees $1,2$, and we
define $H^*_\Iw(G_p,D)$ to be its cohomology, referring to the latter
as the \emph{Iwasawa cohomology} of $D$.  These Iwasawa cohomology
groups are also coadmissible $\calH_E(G_\infty)$-modules. Note the
comparison
\[
\bfR\Ga_\Iw(G_p,V) \cong \bfR\Ga_\Iw(G_p,\bbD_\rig(V)).
\]

We define $\wt\bbD_\cris(D) = D[1/t]^{G_\infty}$ and $\bbD_\cris(D) =
\wt\bbD_\cris(\Hom_{\calR_E}(D,\calR_E))$ (where $t \in \calR$ is
Fontaine's $2\pi i$), and we say that $D$ is crystalline if $\dim_E
\bbD_\cris(D) = \rank_{\calR_E} D$.  Note the comparisons
\[
\bbD_\cris(V) \cong \bbD_\cris(\bbD_\rig(V)),
\qquad
\wt\bbD_\cris(V) \cong \wt\bbD_\cris(\bbD_\rig(V)).
\]
The functor $\wt\bbD_\cris$ provides an exact, rank-preserving
equivalence of exact $\otimes$-categories with Harder--Narasimhan
filtrations, from crystalline $(\vphi,G_\infty)$-modules over
$\calR_E$ to filtered $\vphi$-modules over $E$, under which those
$(\vphi,G_\infty)$-modules of the form $\bbD_\rig(V)$ correspond to
the weakly admissible filtered $\vphi$-modules.  In particular, if we
tacitly equip any $E[\vphi]$-submodule of a filtered $\vphi$-module
with the induced filtration, then for $D$ crystalline $\wt\bbD_\cris$
induces a functorial, order-preserving bijection
\[
\{\text{$t$-saturated $(\vphi,G_\infty)$-submodules of }D\}
\leftrightarrow
\{\text{$E[\vphi]$-stable subspaces of }\wt\bbD_\cris(D)\}.
\]

In the remainder of this subsection, we assume given a continuous
$E$-representation $V$ of $G_{\bbQ,S}$ that is crystalline at $p$, as
well as a fixed $E[\vphi]$-stable $F \subseteq \bbD_\cris(V|_{G_p})$,
and we associate to these data an Iwasawa-theoretic Selmer complex.

We begin by defining a local condition for each $v \in S$, by which we
mean an object $U_v$ in the derived category together with a morphism
$i_v \cn U_v \to \bfR\Ga_\Iw(G_v,V)$.  If $v \neq p$, we denote by
$I_v \subset G_v$ the inertia subgroup, and we let $U_v =
\bfR\Ga_\Iw(G_v/I_v,V^{I_v})$ and $i_v$ be the inflation map.  If $v =
p$, we write $F^\perp \subseteq \wt\bbD_\cris(V)$ for the orthogonal
complement of $F$, and then $D^+_F := \wt\bbD_\cris^{-1}(F^\perp)
\subseteq \bbD_\rig(V)$ and $D^-_F = \bbD_\rig(V)/D^+_F$.  Then we let
$U_v = \bfR\Ga_\Iw(G_p,D^+_F)$, and we let $i_v$ be the functorial map
to $\bfR\Ga_\Iw(G_p,\bbD_\rig(V)) \cong \bfR\Ga_\Iw(G_p,V)$.

We now define the \emph{Selmer complex} $\bfR\wt\Ga_{F,\Iw}(\bbQ,V)$
to be the mapping fiber of the morphism
\[
\bfR\Ga_\Iw(G_{\bbQ,S},V)
\oplus
\bigoplus_{v \in S} U_v
\xrightarrow{\bigoplus_{v \in S} \res_v - \bigoplus_{v \in S} i_v}
\bigoplus_{v \in S} \bfR\Ga_\Iw(G_v,V),
\]
where $\res_v \cn \bfR\Ga_\Iw(G_{\bbQ,S},X) \to \bfR\Ga_\Iw(G_v,X)$
denotes restriction of cochains to the decomposition group.  We write
$\wt H^*_{F,\Iw}(\bbQ,V)$ for its cohomology, referring to it as the
\emph{extended Selmer groups}.  Then $\bfR\wt\Ga_{F,\Iw}(\bbQ,V)$ is a
perfect complex of $\calH_E(G_\infty)$-modules for the range $[0,3]$.

We will have need for a version without imposing local conditions at
$p$.  Namely, we write $\bfR\wt\Ga_{(p),\Iw}(\bbQ,V)$ for the mapping
fiber of
\[
\bfR\Ga_\Iw(G_{\bbQ,S},V)
\oplus
\bigoplus_{v \in S^{(p)}} U_{v,\Iw}
\xrightarrow{\bigoplus_{v \in S^{(p)}} \res_v
  - \bigoplus_{v \in S^{(p)}} i_v}
\bigoplus_{v \in S^{(p)}} \bfR\Ga_\Iw(G_v,V),
\]
where $S^{(p)} = S \bs \{p\}$, and we write $\wt
H^*_{(p),\Iw}(\bbQ,V)$ for its cohomology.  Bearing in mind the exact
triangle
\[
\bfR\Ga_\Iw(G_p,D^+_F) \to \bfR\Ga_\Iw(G_p,V) \to
\bfR\Ga_\Iw(G_p,D^-_F) \to \bfR\Ga_\Iw(G_p,D^+_F)[1],
\]
we deduce from the definitions of the Selmer complexes an exact
triangle
\begin{equation}\label{E:remove-p}
\bfR\wt\Ga_{F,\Iw}(\bbQ,V) \to \bfR\wt\Ga_{(p),\Iw}(\bbQ,V) \to
\bfR\Ga_\Iw(G_p,D^-_F) \to \bfR\wt\Ga_{F,\Iw}(\bbQ,V)[1].
\end{equation}

\subsection{The Main Conjecture for $f$ and its symmetric powers}

We remind the reader of the fixed newform $f$ of weight $k$, level
$\Ga_1(N)$ with $p \nmid N$ and character $\ep$, with CM by $K$, and
the roots $\alpha,\ov\alpha$ of $x^2 + \ep(p)p^{k-1}$.

Since the elements $\alpha,\ov\alpha$ are distinct, the
$\vphi$-eigenspace with eigenvalue $\alpha$ determines an
$E[\vphi$]-stable subspace $F_\alpha \subseteq \bbD_\cris(V_f)$.  We
apply the constructions of Iwasawa-theoretic extended Selmer groups,
with their associated ranks and characteristic ideals, to the data of
$V_f$ equipped with $F_\alpha$.

The following is the ``finite-slope'' form of the Main Conjecture of
Iwasawa theory for $f$.

\begin{theorem}\label{t:MC}
Assume that $p$ does not divide the order of the nebentypus $\ep$.
The coadmissible $\calH_E(G_\infty)$-module $\wt
H^2_{F_\alpha,\Iw}(\bbQ,V_f)$ is torsion, and
\[
\chr_{\calH_E(G_\infty)} \wt H^2_{F_\alpha,\Iw}(\bbQ,V_f)
=
(\Tw_{-1} L_{f,\alpha}).
\]
\end{theorem}

\begin{proof}
We reproduce the argument of \cite[\S5]{P2}, adapted to the
normalizations of this paper.

In the notation of \S\ref{S:selmer}, the object $D^-_{F_\alpha}$ is
crystalline, and $\wt\bbD_\cris(D^-_{F_\alpha})$ has
$\vphi$-eigenvalue $\alpha\inv$ and Hodge--Tate weight $0$.  This
implies that $H^2_\Iw(G_p,D^-_{F_\alpha}) = 0$.  (If $k$ is odd,
$\ep(p)=-1$, and $\alpha=+p^{(k-1)/2}$ then
$H^1_\Iw(G_p,D^-_{F_\alpha})_\tors \cong E(\chi^{(k-1)/2})$ is
nonzero, but this ``exceptional zero'' does not affect the present
proof.)

Write $f^c = f \otimes \ep\inv$ for the eigenform with Fourier
coefficients complex conjugate to those of $f$, and recall the duality
$\Hom_E(V_{f^c},E) \cong V_f(1-k)$.  Let $z'_{f^c} \in \wt
H^1_{(p),\Iw}(\bbQ,\Hom_E(V_{f^c},E))$ denote Kato's zeta element
derived from elliptic units (denoted $z_\ga^{(p)}(f^*)$ for suitable
$\ga \in \Hom_E(V_{f^c},E)$ in \cite{K}), and let
\[
z_f = \Tw_{k-1} z'_{f^c} \in
\wt H^1_{(p),\Iw}(\bbQ,\Hom_E(V_{f^c},E)(k-1)) \cong
\wt H^1_{(p),\Iw}(\bbQ,V_f).
\]
For a crystalline $(\vphi,G_\infty)$-module $D$ satisfying
$\Fil^1\bbD_\dR(D) = 0$, recall the dual of the big exponential map
treated in \cite[\S3]{Nak}:
\[
\Exp^*_{D^*}
\cn
H^1_\Iw(G_p,D)
\to
\wt\bbD_\cris(D) \otimes_E \calH_E(G_\infty).
\]
By naturality in $D$, there is a commutative diagram
\[\begin{array}{r@{\ }c@{\ }c@{\ }c}
\wt H^1_{(p),\Iw}(\bbQ,V_f)
  \xrightarrow{\loc_{V_f}} H^1_\Iw(G_p,V_f)
  \cong
& H^1_\Iw(G_p,\bbD_\rig(V_f))
& \xrightarrow{\Col_\alpha} &
  H^1_\Iw(G_p,D^-_{F_\alpha}) \\
& \Exp^*_{V_f^*}\downarrow\phantom{\Exp^*_{V_f^*}} &
& \phantom{\Exp^*_{D^{-,*}_{F_\alpha}}} \downarrow \Exp^*_{D^{-,*}_{F_\alpha}} \\
& \wt\bbD_\cris(V_f) \otimes_E \calH_E(G_\infty)
& \to &
  \wt\bbD_\cris(D^-_{F_\alpha}) \otimes_E \calH_E(G_\infty).
\end{array}\]
Write $\loc_\alpha = \Col_\alpha \circ \loc_{V_f}$, where the maps
$\loc_{V_f}$ and $\Col_\alpha$ are as in the preceding diagram.
Identifying $\wt\bbD_\cris(D_{F_\alpha}^-) = \Hom_E(Ee_\alpha,E)$,
\cite[Theorem~16.6(2)]{K} shows that
\begin{equation}\label{E:compute-kato}
(\Tw_1 \Exp^*_{D^{-,*}_{F_\alpha}} \loc_\alpha z_f)(e_\alpha)
=
(\Exp^*_{\Hom_E(V_f,E)} \loc_{V_f(1)} \Tw_1 z_f)(e_\alpha)
=
L_{f,\alpha},
\end{equation}
after perhaps rescaling $e_\alpha$.  In particular, $\loc_\alpha$ is a
nontorsion morphism.

The exact triangle \eqref{E:remove-p} gives rise to an exact sequence
\begin{multline*}
0
\to
\wt H^1_{F_\alpha,\Iw}(\bbQ,V_f)
\to
\wt H^1_{(p),\Iw}(\bbQ,V_f)
\xrightarrow{\loc_\alpha}
H^1_\Iw(G_p,D^-_{F_\alpha}) \\
\to
\wt H^2_{F_\alpha,\Iw}(\bbQ,V_f)
\to
\wt H^2_{(p),\Iw}(\bbQ,V_f)
\to
0.
\end{multline*}
It follows from \cite[Theorem~12.4]{K} that the finitely generated
$\calH_E(G_\infty)$-module $\wt H^1_{(p),\Iw}(\bbQ,V_f)$ (resp.\ $\wt
H^2_{(p),\Iw}(\bbQ,V_f)$) is free of rank $1$ (resp.\ is torsion).
Employing the local Euler--Poincar\'e formula and the fact that
$\loc_\alpha$ is nontorsion, we see from the preceding exact sequence
that $\wt H^1_{F_\alpha,\Iw}(\bbQ,V_f) = 0$, $\wt
H^2_{F_\alpha,\Iw}(\bbQ,V_f)$ is torsion, and
\begin{multline*}
\left(\chr_{\calH_E(G_\infty)}
  \frac{\wt H^1_{(p),\Iw}(\bbQ,V_f)}{\calH_E(G_\infty)z_f}\right)
\left(\chr_{\calH_E(G_\infty)} \wt H^2_{F_\alpha,\Iw}(\bbQ,V_f)\right) \\
=
\left(\chr_{\calH_E(G_\infty)}
  \frac{H^1_\Iw(G_p,D^-_{F_\alpha})}{\calH_E(G_\infty)\loc_\alpha z_f}\right)
\left(\chr_{\calH_E(G_\infty)} \wt H^2_{(p),\Iw}(\bbQ,V_f)\right).
\end{multline*}
Applying $\Tw_{k-1}$ to the claim of \cite[Theorem~12.5(3)]{K} with
$f^*$ in place of $f$, we deduce that
\[
\chr_{\calH_E(G_\infty)}
  \frac{\wt H^1_{(p),\Iw}(\bbQ,V_f)}{\calH_E(G_\infty)z_f}
=
\chr_{\calH_E(G_\infty)} \wt H^2_{(p),\Iw}(\bbQ,V_f),
\]
and therefore
\[
\chr_{\calH_E(G_\infty)} \wt H^2_{F_\alpha,\Iw}(\bbQ,V_f)
=
\chr_{\calH_E(G_\infty)}
  \frac{H^1_\Iw(G_p,D^-_{F_\alpha})}{\calH_E(G_\infty)\loc_\alpha z_f}.
\]
Although only a divisibility of characteristic ideals is claimed by
Kato, one easily checks that his proof, especially
\cite[Proposition~15.17]{K}, gives an equality whenever Rubin's method
gives an equality.  Under the hypothesis that $\ep$ has order prime to
$p$, the required extension of Rubin's work is precisely
Theorem~\ref{T:rubin} below.  It remains to compute the right hand
side of the last identity.  In fact, one has the exact sequence
\begin{multline*}
0 \to
H^1_\Iw(G_p,D^-_{F_\alpha})_\tors
\to
\frac{H^1_\Iw(G_p,D^-_{F_\alpha})}{\calH_E(G_\infty)\loc_\alpha z_f} \\
\xrightarrow{\Exp^*_{D^{-,*}_{F_\alpha}}}
\frac{\wt\bbD_\cris(D^-_{F_\alpha}) \otimes_E \calH_E(G_\infty)}
  {\calH_E(G_\infty)\Exp^*_{D^{-,*}_{F_\alpha}}\loc_\alpha z_f}
\to
\coker \Exp^*_{D^{-,*}_{F_\alpha}}
\to 0,
\end{multline*}
and because $D^-_{F_\alpha}$ has Hodge--Tate weight zero and
$H^2_\Iw(G_p,D^-_{F_\alpha})=0$, \cite[Theorem~3.21]{Nak} shows that
\[
\chr_{\calH_E(G_\infty)} H^1_\Iw(G_p,D^-_{F_\alpha})_\tors
=
\chr_{\calH_E(G_\infty)} \coker \Exp^*_{D^{-,*}_{F_\alpha}},
\]
and hence
\[
\chr_{\calH_E(G_\infty)}
  \frac{H^1_\Iw(G_p,D^-_{F_\alpha})}{\calH_E(G_\infty)\loc_\alpha z_f}
=
\chr_{\calH_E(G_\infty)}
  \frac{\wt\bbD_\cris(D^-_{F_\alpha}) \otimes_E \calH_E(G_\infty)}
    {\calH_E(G_\infty)\Exp^*_{D^{-,*}_{F_\alpha}}\loc_\alpha z_f}.
\]
Finally, \eqref{E:compute-kato} shows that the right hand side above
is generated by $\Tw_{-1} L_{f,\alpha}$.
\end{proof}

We now turn to the Main Conjecture of Iwasawa theory for $V_m$ in its
``finite-slope'' form, beginning with two remarks.  First, we remind
the reader that since $\Sel_{k_\infty}(A_{\ep_K^r}^*)^\vee$ is a
finitely generated, torsion $\La_{\calO_E}(G_\infty)$-module, it
follows that
\[
\Sel_{k_\infty}(A_{\ep_K^r}^*)^\vee[1/p]
=
\Sel_{k_\infty}(A_{\ep_K^r}^*)^\vee
  \otimes_{\La_{\calO_E}(G_\infty)} \La_E(G_\infty)
\stackrel\sim\to
\Sel_{k_\infty}(A_{\ep_K^r}^*)^\vee
  \otimes_{\La_{\calO_E}(G_\infty)} \calH_E(G_\infty),
\]
and therefore $\Sel_{k_\infty}(A_{\ep_K^r}^*)^\vee[1/p]$ is naturally
a finitely generated (hence coadmissible), torsion
$\calH_E(G_\infty)$-module.  Second, for $\fks \in \fkS$ we note that
the ``plus and minus'' Iwasawa-theoretic Selmer groups satisfy the
arithmetic duality
\[
H^{1,\fks_i}_f(k_\infty,A_{f_i}^*((r-i)(k-1)))^\vee[1/p]
\cong
\wt H^2_{\fks_i,\Iw}(\bbQ,V_{f_i}((i-r)(k-1)))^\iota,
\]
where $\wt H^2_{\fks_i,\Iw}$ denotes the cohomology of an
Iwasawa-theoretic Selmer complex with local condition at $p$
appropriately built from the choice $\fks_i$.  These isomorphic
modules are also finitely generated (hence coadmissible), torsion
$\calH_E(G_\infty)$-modules, by Theorem~5.6.

With the preceding remarks in mind, what follows is the finite-slope
analogue of Definition~5.3.  Fix $\fkt = (\fkt_0,\ldots,\fkt_{\wt
  r-1}) \in \fkT$.  For each $i=0,\ldots,\wt r-1$, the elements
$\alpha_i,\ov\alpha_i$ are distinct, so the $\vphi$-eigenspace with
eigenvalue $\fkt_i p^{(r-i)(k-1)}$ determines an $E[\vphi]$-stable
subspace $F_i \subseteq \bbD_\cris(V_{f_i}((i-r)(k-1)))$.  We may
apply the constructions of Iwasawa-theoretic extended Selmer groups,
with their associated ranks and characteristic ideals, to the data of
$V_{f_i}((i-r)(k-1))$ equipped with $F_i$.

\begin{definition}\label{d:Selmer}
For $\fkt \in \fkT$, we define the coadmissible
$\calH_E(G_\infty)$-module
\[
\Sel_{k_\infty}^\fkt(V_m^*)^\vee
:=
\left(\bigoplus_{i=0}^{\wt r-1}
  \wt H^2_{F_i,\Iw}\left(\bbQ,V_{f_i}((i-r)(k-1))\right)^\iota\right)
\oplus
\begin{cases}
\Sel_{k_\infty}(A_{\ep_K^r}^*)^\vee[1/p] & m\text{ even}, \\
0 & m\text{ odd}.
\end{cases}
\]
\end{definition}

\begin{remark}
Although the notation $\Sel_{k_\infty}^\fkt(V_m^*)^\vee$ in the
finite-slope case was chosen for symmetry with
$\Sel_{k_\infty}^\fks(A_m^*)^\vee[1/p]$ in the ``plus and minus''
case, this notation is highly misleading: it is an essential feature
of the finite-slope theory that $\Sel_{k_\infty}^\fkt(V_m^*)^\vee$ is
coadmissible but typically \emph{not} finitely generated over
$\calH_E(G_\infty)$, and therefore does not arise as the Pontryagin
dual (with $p$ inverted) of direct limits of finite-layer objects, as
$\Sel_{k_\infty}^\fks(A_m^*)^\vee[1/p]$ does.  This fact forces us to
work on the other side of arithmetic duality, as in the first summand
above.
\end{remark}

\begin{theorem}\label{t:MC2}
For all $\fkt \in \fkT$, the coadmissible $\calH_E(G_\infty)$-module
$\wt H^2_{\fkt,\Iw}(\bbQ,V_m)$ is torsion, and
\[
\chr_{\calH_E(G_\infty)} \Sel_{k_\infty}^\fkt(V_m^*)^\vee
=
(\Tw_1 \wt L_{V_m,\fkt}).
\]
\end{theorem}

\begin{proof}
Just as in the proof of Theorem~5.9, this theorem follows from
Theorem~5.5 and from Theorem~\ref{t:MC} applied to each $f_i$.
\end{proof}

\section{The Main Conjecture for imaginary quadratic fields at inert
  primes}

In the fundamental works \cite{R,R2}, Rubin perfected the Euler system
method for elliptic units.  From this he deduced a divisibility of
characteristic ideals as in the Main Conjecture of Iwasawa theory.  In
most cases, he used the analytic class number formula to promote the
divisibilities to identities.  In this section we extend the use of
the analytic class number formula to the remaining cases.  The
obstruction in these problematic cases is that the control maps on
global/elliptic units and class groups are far from being
isomorphisms.  Our approach is to use base change of Selmer complexes
to get a precise description of the failure of control, and then to
apply a characterization of $\mu$- and $\la$-invariants that is valid
even in the presence of zeroes of the characteristic ideal at
finite-order points.  This section is written independently of the
preceding notations and hypotheses of this paper and \cite{HL}; we
employ notations as in \cite{R}, recalled as follows.

We take $K$ to be an imaginary quadratic field, and $p$ an odd prime
inert in $K$.  Let $K_0$ be a finite abelian extension with $\Delta =
\Gal(K_0/K)$ and $\de = [K_0:K]$, and assume that $p \nmid \de$.  Let
$K_\infty$ be an abelian extension of $K$ containing $K_0$, such that
$\Ga = \Gal(K_\infty/K_0)$ is isomorphic to $\bbZ_p$ or $\bbZ_p^2$.
One has $\scrG = \Gal(K_\infty/K) = \Delta \times \Ga$.  Accordingly,
$K_\infty = K_0 \cdot K_\infty^\Delta$, where
$\Gal(K_\infty^\Delta/K)$ is identified with $\Ga$.

We write $\La = \La(\scrG) =\bbZ_p[\![\scrG]\!]$.  The letter $\eta$
will always range over the irreducible $\bbZ_p$-representations of
$\Delta$.  One has $\bbZ_p[\Delta] = \bigoplus_\eta
\bbZ_p[\Delta]^\eta$, where $\bbZ_p[\Delta]^\eta$ is isomorphic to the
ring of integers in the unramified extension of $\bbQ_p$ of degree
$\dim(\eta)$, and, accordingly, $\La = \bigoplus_\eta
\bbZ_p[\Delta]^\eta[\![\Ga]\!]$.  The sum map $\summ \cn
\bbZ_p[\Delta] \to \bbZ_p$, $\sum_\sigma n_\sigma\sigma \mapsto
\sum_\sigma n_\sigma$, is identified with the projection onto the
component $\bbZ_p[\Delta]^{\bf1}$ indexed by the trivial character
${\bf1}$; write $\bbZ_p[\Delta]^!$ for the kernel of the sum map,
which is equal to $\bigoplus_{\eta\neq{\bf1}} \bbZ_p[\Delta]^\eta$,
and satisfies $\bbZ_p[\Delta] = \bbZ_p[\Delta]^{\bf1} \oplus
\bbZ_p[\Delta]^!$.

For $\{a_n\}$ a sequence of positive real numbers, if there exist real
numbers $\mu,\la$ such that $\log_p a_n = \mu p^n + \la n + O(1)$ as
$n \to +\infty$, then these numbers $\mu,\la$ are uniquely determined
by $\{a_n\}$, and we write $\mu = \mu(\{a_n\})$ and $\la =
\la(\{a_n\})$.

\begin{lemma}\label{L:numerics}
Assume that $\Ga$ is isomorphic to $\bbZ_p$, and let $M$ be a finitely
generated, torsion $\bbZ_p[\![\Ga]\!]$-module.  Then for $n \gg 0$ the
quantity $\rank_{\bbZ_p} M_{\Ga^{p^n}}$ stabilizes to some integer $r
\geq 0$, so that $M_{\Ga^{p^n}} \approx \bbZ_p^{\oplus r} \oplus
M_{\Ga^{p^n}}[p^\infty]$, and Iwasawa's $\mu$- and $\la$-invariants of
$M$ satisfy $\mu(M) = \mu(\{\#M_{\Ga^{p^n}}[p^\infty]\})$ and $\la(M)
= r + \la(\{\#M_{\Ga^{p^n}}[p^\infty]\})$.
\end{lemma}

\begin{proof}
One easily sees that if $M \to M'$ is a pseudo-isomorphism, then both
sides of the desired identities are invariant under replacing $M$ by
$M'$.  Using the structure theorem and additivity over direct sums, it
therefore suffices to check the case where $M = \bbZ_p[\![\Ga]\!]/(f)$
for prime $f \in \bbZ_p[\![\Ga]\!]$.  The case where $f$ is relatively
prime to all the augmentation ideals $I(\Ga^{p^k}) = (f_k)$ of
$\Ga^{p^k}$ for $k \geq 0$, or equivalently where $r=0$, is
well-known.  The remaining case is where $f = f_k/f_{k-1}$ for $k \geq
0$ (we set $f_{-1}=1$), whence one has
\[
(\bbZ_p[\![\Ga]\!]/(f))_{\Ga^{p^n}} = \bbZ_p[\![\Ga]\!]/(f,f_n) =
\bbZ_p[\![\Ga]\!]/(f) \approx \bbZ_p^{\oplus (p-1)p^{k-1}}
\]
for $n \geq k$, agreeing with the Iwasawa invariants.
\end{proof}

Let $F$ be a subextension of $K_\infty/K_0$.  If $F/K_0$ is finite, we
associate to it the following objects:
\begin{itemize}
\item $A(F) = \Pic(\calO_F) \otimes_\bbZ \bbZ_p$ is the $p$-part of
  its ideal class group,
\item $X(F) = \Pic(\calO_F,p^\infty) = \llim_n (\Pic(\calO_F,p^n)
  \otimes_\bbZ \bbZ_p)$ is the inverse limit of the $p$-parts of its
  ray class groups of conductor $p^n$,
\item $U(F) = (\calO_F \otimes_\bbZ \bbZ_p)^\times_\prop$ is the
  pro-$p$ part of its group of semilocal units,
\item $\scrE(F) = \calO_F^\times \otimes_\bbZ \bbZ_p$ is its group of
  global units $\otimes \bbZ_p$, and
\item $\scrC(F)$ is its group of elliptic units $\otimes \bbZ_p$, as
  defined in \cite[\S1]{R}.
\end{itemize}
If $F/K_0$ is infinite, and $? \in \{A,X,U,\scrE,\scrC\}$, we let
$?(F) = \llim_{F_0} ?(F_0)$, where $F_0$ ranges over the finite
subextensions of $F$, obtaining a finitely generated
$\bbZ_p[\![\Gal(F/K)]\!]$-module.  Note that Leopoldt's conjecture is
known in this case, so by the definition of ray class groups one has a
short exact sequence
\[
0 \to \scrE(F) \to U(F) \to X(F) \to A(F) \to 0.
\]
Class field theory identifies $A(F)$ (resp.\ $X(F)$) with the Galois
group of the maximal $p$-abelian extension of $F$ which is everywhere
unramified (resp.\ unramified at primes not dividing $p$).

The following improvement of Rubin's work is the main result of this
section, and the remainder of this section consists of its proof.

\begin{theorem}\label{T:rubin}
One has the equality of characteristic ideals,
\[
\chr_\La A(K_\infty) = \chr_\La(\scrE(K_\infty)/\scrC(K_\infty)).
\]
\end{theorem}

In \cite[Theorem~4.1(ii)]{R} and \cite[Theorem~2(ii)]{R2} it is proved
that both sides are nonzero at each $\eta$-factor, that $\chr_\La
A(K_\infty)$ divides $\chr_\La(\scrE(K_\infty)/\scrC(K_\infty))$, and
that the $\eta$-factors are equal when $\eta$ is nontrivial on the
decomposition group of $p$ in $\Delta$.  To get equality for the
remaining $\eta$, we may thus reduce to the case where $p$ is totally
split in $K_0/K$.  We also specialize our notation to where
$K_\infty^\Delta$ is any $\bbZ_p^1$-extension.  We index finite
subextensions $F$ of $K_\infty/K_0$ as $F = K_n =
K_\infty^{\Ga^{p^n}}$ for $n \geq 0$.  Fix a topological generator
$\ga \in \Ga$, and for brevity write $\La_n = \bbZ_p[\scrG/\Ga^{p^n}]
= \La/(\ga^{p^n}-1)$.

There is unique $\bbZ_p^2$-extension of $K$, and it contains all
$\bbZ_p$-extensions of $K$.  This extension is unramified at all
primes not dividing $p$, and Lubin--Tate theory shows it is totally
ramified at $p$.  The same ramification behavior is true of any
$\bbZ_p$-extension, as well as of $K_\infty/K_0$ because $p$ is
totally split in $K_0/K$.  In particular, if $S_n$ denotes the set of
places of $K_n$ lying over $p$, then the restriction maps $S_{n+1} \to
S_n$ are bijections, and $S_n$ is a principal homogeneous
$\Delta$-set.  Fixing once and for all $v_0 \in S_0$, with unique lift
$v_n \in S_n$, declaring $v_n$ to be a basepoint of $S_n$ gives an
identification $\bbZ_p[S_n] \cong \bbZ_p[\Delta]$ of
$\bbZ_p[\Delta]$-modules.  We write $\invt$ for the composite of the
semilocal restriction map, the invariant maps of local class field
theory, and this identification:
\[
\invt \cn
H^2(G_{K,\{p\}},\bbZ_p(1))
\to
\bigoplus_{v \in S_n} H^2(G_{K_{n,v}},\bbZ_p(1))
\cong
\bbZ_p[S_n]
\cong
\bbZ_p[\Delta].
\]
Also, it follows that $p-1$ does not divide the ramification degree of
$K_\infty/\bbQ$ at $p$, so that $\mu_{p^\infty}(K_{\infty,v}) = 1$ for
any place $v$ of $K_\infty$ lying over $p$.  Therefore, for $F/K_0$
finite the group $(\calO_F \otimes_\bbZ \bbZ_p)^\times$ is already
pro-$p$.

Since $\chr_\La A(K_\infty)$ divides
$\chr_\La(\scrE(K_\infty)/\scrC(K_\infty))$, their Iwasawa $\mu$- and
$\la$-invariants, considered as a $\bbZ_p[\![\Ga]\!]$-modules, satisfy
\begin{equation}\label{E:rubin}
\mu(A(K_\infty)) \leq \mu(\scrE(K_\infty)/\scrC(K_\infty)),
\qquad
\la(A(K_\infty)) \leq \la(\scrE(K_\infty)/\scrC(K_\infty)),
\end{equation}
We shall improve these inequalities to the claim that for some $\ep
\in \{0,1\}$ one has
\[
\mu(A(K_\infty)) = \mu(\scrE(K_\infty)/\scrC(K_\infty)),
\qquad
\ep + \la(A(K_\infty)) = \la(\scrE(K_\infty)/\scrC(K_\infty)),
\]
and additionally
\[
\rank_{\bbZ_p} A(K_\infty)_\scrG = 0,
\qquad
\rank_{\bbZ_p} (\scrE(K_\infty)/\scrC(K_\infty))_\scrG = \ep.
\]
These computations are equivalent to the claim that
\begin{equation}\label{E:subtheorem}
(\chr_\La \bbZp)^\ep \cdot \chr_\La A(K_\infty)
= \chr_\La \scrE(K_\infty)/\scrC(K_\infty).
\end{equation}
Granted \eqref{E:subtheorem}, let us show how to deduce the theorem.
Let $K'_\infty$ denote the compositum of $K_0$ with the unique
$\bbZ_p^2$-extension of $K$, and write $\scrG' = \Gal(K'_\infty/K) =
\Delta \times \Ga'$, $\La' = \La(\scrG') = \bbZ_p[\![\scrG']\!]$, and
$\proj \cn \La' \twoheadrightarrow \La$.  By Rubin's theorem, there
exist $f' \in \La(\scrG')$ and $f \in \La(\scrG)$ with
\[
f' \cdot \chr_{\La'} A(K'_\infty)
=
\chr_{\La'} \scrE(K'_\infty)/\scrC(K'_\infty),
\quad
f \cdot \chr_\La A(K_\infty)
=
\chr_\La \scrE(K_\infty)/\scrC(K_\infty).
\]
By \cite[Corollary 7.9(i)]{R} one has $\proj(f') = f$ up to a unit in
$\La$.  Since $\proj$ is a homomorphism of semilocal rings that is a
bijection on local factors and restricts to a local homomorphism on
each local factor, it follows that $f'$ is a unit (resp.\ restricts to
a unit over a given local factor) in $\La'$ if and only if $f$ is a
unit (resp.\ restricts to a unit over the corresponding local factor)
in $\La$.  On the other hand, \eqref{E:subtheorem} implies that $f$
divides $\chr_\La \bbZ_p$ in $\La$.  Since $(\chr_\La \bbZ_p)^\eta =
\La^\eta$, the unit ideal, if $\eta \neq {\bf1}$, we deduce the
identity of the theorem for both $\bbZ_p^1$- and $\bbZ_p^2$-extensions
over each such $\eta$-factor.  We only have left to consider the case
where $\eta={\bf1}$, or rather where $\Delta$ is trivial and $K_0=K$.

\begin{lemma}
Write $R = \bbZ_p[\![S,T]\!]$, and for $a,b \in \bbZ_p$ not both
divisible by $p$, write $R_{a,b} = R/((1+S)^a(1+T)^b-1)$ with
$\pi_{a,b} \cn R \twoheadrightarrow R_{a,b}$.  We identify $R_{a,b}
\cong \bbZ_p[\![U]\!]$, where $U = \pi_{a,b}(S)$ if $p \nmid b$ and
$U = \pi_{a,b}(T)$ otherwise.

Suppose $g \in R$ is such that for all $a,b$ above, $\pi_{a,b}(g)$
divides $U$ in $R_{a,b}$.  Then $g$ is a unit.
\end{lemma}

\begin{proof}
Write $g = x + yS + zT + O((S,T)^2)$ with $x,y,z \in \bbZ_p$; we are
to show that $p \nmid x$.  Since $\pi_{0,1}(g)$ divides $U$ in
$R_{0,1}$, and $R_{0,1}$ is a UFD with $U$ a prime element, it follows
that $\pi_{0,1}(g)$ is either a unit or $U$ times a unit.  As
$\pi_{0,1}(g) = x + yU + O(U^2)$, the first case is equivalent to $p
\nmid x$, and the second case is equivalent to $x=0$ and $p \nmid y$.
But in the second case the identity
\[
g = yS + zT + O((S,T)^2) = (1+S)^y(1+T)^z-1 + O((S,T)^2)
\]
would imply $\pi_{y,z}(g) = 0 + O(U^2)$, that is $U^2$ divides
$\pi_{y,z}(g)$ in $R_{y,z}$, contradicting that $\pi_{y,z}(g)$ divides
$U$.
\end{proof}

Choose a $\bbZ_p$-basis $\ga_1,\ga_2 \in \Ga'$, so that $\ker(\Ga'
\twoheadrightarrow \Ga) = (\ga_1^a\ga_2^b)^{\bbZ_p}$ for some $a,b \in
\bbZ_p$ not both divisible by $p$.  Set $S=\ga_1-1,T=\ga_2-1 \in
\La'$, and note that $\ker(\La' \twoheadrightarrow \La)$ is generated
by $(1+S)^a(1+T)^b-1$, so that the map $\La' \twoheadrightarrow \La$
is identified with the map $\pi_{a,b} \cn R \twoheadrightarrow
R_{a,b}$ of the preceding lemma.  Under this identification, the
augmentation ideal $\chr_\La \bbZ_p$ is generated by $U \in R_{a,b}$,
so we have that $\pi_{a,b}(f') = f$ divides $U$.  Since
$K_\infty^\Delta = K_\infty$ was allowed to be any $\bbZ_p$-extension
of $K$, and conversely every such pair of $a,b$ arises from some
choice of $K_\infty$, the preceding lemma shows that $f'$ is a unit,
and therefore so is $f$, proving the theorem at once for $\bbZ_p^1$-
and $\bbZ_p^2$-extensions.

We begin the proof of \eqref{E:subtheorem} (and no longer assume that
$K_0=K$).  As mentioned at the beginning of this section, our approach
is to use base change of Selmer complexes to measure the failure of
the maps $(\scrE(K_\infty)/\scrC(K_\infty))_{\Ga^{p^n}} \to
\scrE(K_n)/\scrC(K_n)$ and $A(K_\infty)_{\Ga^{p^n}} \to A(K_n)$ to be
isomorphisms.

Since we use base change in the derived category, we give some
generalities on the operation $\Lotimes_\La \La_n$.  We first compute
that $\La_n[0] \cong [\La \xrightarrow{\ga^{p^n}-1} \La]$ as objects
in the derived category of $\La$-modules, the latter concentrated in
degrees $-1,0$, so that for any $\La$-module (resp.\ complex of
$\La$-modules) $X$ one may compute $X \Lotimes_\La \La_n$ as $[X
  \xrightarrow{\ga^{p^n}-1} X]$ (resp.\ as the mapping cone of
$\ga^{p^n}-1$ on $X$).  The induced map $X \Lotimes_\La \La_{n+1} \to
X \Lotimes_\La \La_n$ corresponds to the morphism $[X
  \xrightarrow{\ga^{p^{n+1}}-1} X] \to [X \xrightarrow{\ga^{p^n}-1}
  X]$ given by multiplication by $1+\ga^{p^n}+\cdots+\ga^{(p-1)p^n}$
in shift degree $-1$, and by the identity in shift degree $0$.
Alternatively, the $\Tor$ spectral sequence degenerates to short exact
sequences
\begin{equation}\label{E:generic-base-change}
0 \to
H^i(X)_{\Ga^{p^n}}
\to
H^i(X \Lotimes_\La \La_n)
\to
H^{i+1}(X)^{\Ga^{p^n}}
\to 0,
\end{equation}
and the natural morphism from the above sequence for $n+1$ to the
sequence for $n$ is given by the natural projection on the first term,
and by multiplication by $1+\ga^{p^n}+\cdots+\ga^{(p-1)p^n}$ on the
last term.  The Bockstein homomorphism $\beta = \beta_X$, defined as
the connecting homomorphism in the exact triangle
\begin{multline*}
X \Lotimes_\La
\left(\La_n \xrightarrow{\ga^{p^n}-1} \La/(\ga^{p^n}-1)^2
  \to \La_n \to \La_n[1]\right) \\
\cong
\left(X \Lotimes_\La \La_n
\xrightarrow{\ga^{p^n}-1}
X \Lotimes_\La \La/(\ga^{p^n}-1)^2
\to
X \Lotimes_\La \La_n
\xrightarrow\beta
X \Lotimes_\La \La_n[1]\right),
\end{multline*}
is computed on cohomology as the composite
\begin{multline*}
H^i(\beta) \cn
H^i(X \Lotimes_\La \La_n)
\twoheadrightarrow
H^i(X \Lotimes_\La \La_n)/H^i(X)_{\Ga^{p^n}} \\
\cong
H^{i+1}(X)^{\Ga^{p^n}}
\hookrightarrow
H^{i+1}(X)
\twoheadrightarrow
H^{i+1}(X)_{\Ga^{p^n}}
\hookrightarrow
H^{i+1}(X \Lotimes_\La \La_n).
\end{multline*}
Note that if $Z$ is a finitely generated, torsion $\La$-module, then
$\rank_{\bbZ_p} Z^{\Ga^{p^n}} = \rank_{\bbZ_p} Z_{\Ga^{p^n}}$.

If $X$ satisfies $X = X^\Ga$, then the above computations reduce to $X
\Lotimes_\La \La_n \cong X[1] \oplus X$, in such a way that the
natural map $X \Lotimes_\La \La_{n+1} \to X \Lotimes_\La \La_n$ is
identified with multiplication by $p$ in shift degree $-1$, and with
the identity map in shift degree $0$.  The Bockstein homomorphism
\[
\beta \cn
X[1] \oplus X = X \Lotimes_\La \La_n
\to
X \Lotimes_\La \La_n[1] = X[2] \oplus X[1]
\]
is the identity map on $X[1]$ and zero on the other factors.  In this
scenario, we write $\beta\inv = \beta_X\inv \cn X[2] \oplus X[1] \to
X[1] \oplus X$ for the map that is inverse to this identity map on
$X[1]$ and zero on the other factors.  Any morphism $f \cn Y \to X$
gives rise to a morphism $f \Lotimes_\La \La_n \cn Y \Lotimes_\La
\La_n \to X \Lotimes_\La \La_n = X[1] \oplus X$.  Writing $f
\otimes_\La \La_n$ for the projection of $f \Lotimes_\La \La_n$ onto
the second component, $X$, the commutative diagram
\[\begin{array}{ccccc}
Y \Lotimes_\La \La_n
& \xrightarrow{f \Lotimes_\La \La_n} &
X \Lotimes_\La \La_n
& = &
X[1] \oplus X \\
\beta_Y\downarrow\phantom{\beta_Y}
& &
\beta_X\downarrow\phantom{\beta_X}
& &
\phantom\sim\searrow\sim \\
Y \Lotimes_\La \La_n[1]
& \xrightarrow{f \Lotimes_\La \La_n[1]} &
X \Lotimes_\La \La_n[1]
& = &
X[2] \oplus X[1]
\end{array}\]
shows that the projection of $f \Lotimes_\La \La_n$ onto the first
component, $X[1]$, is computed by $\beta_X\inv \circ (f \otimes_\La
\La_n)[1] \circ \beta_Y$.

We now return to the setting of the theorem, recalling \Nekovar's
constructions of the fundamental invariants of number fields in terms
of Selmer complexes in \cite[\S9.2,\S9.5]{N} (with notations adapted
to our situation).  Throughout, $n \geq 0$ ranges over nonnegative
integers.  For brevity we write
\[
\bfR\Ga_n = \bfR\Ga_\cont(G_{K_n,\{p\}},\bbZ_p(1)),
\qquad
\bfR\Ga_\Iw
= \bfR\Ga_\Iw(K_\infty/K_0,\bbZ_p(1))
= \bfR\!\llim_n \bfR\Ga_n,
\]
and $H^i_? = H^i(\bfR\Ga_?)$ for $? \in \{n,\Iw\}$.  Then one
has the computations
\begin{gather*}
H^i_n = 0,\ i \neq 1,2, \qquad
H^1_n = \calO_{K_n,\{p\}}^\times \otimes_\bbZ \bbZ_p, \\
0 \to
\Pic(\calO_{K_n,\{p\}}) \otimes_\bbZ \bbZ_p
\to
H^2_n
\xrightarrow{\invt}
\bbZ_p[\Delta]
\xrightarrow{\summ}
\bbZ_p
\to 0,
\end{gather*}
and, passing to inverse limits (Mittag-Leffler holds by compactness),
\begin{gather*}
H^i_\Iw = 0,\ i \neq 1,2, \qquad
H^1_\Iw = \llim_n (\calO_{K_n,\{p\}}^\times \otimes_\bbZ \bbZ_p), \\
0 \to \llim_n (\Pic(\calO_{K_n,\{p\}}) \otimes_\bbZ \bbZ_p)
\to
H^2_\Iw
\xrightarrow{\invt}
\bbZ_p[\Delta]
\xrightarrow{\summ}
\bbZ_p \to 0.
\end{gather*}
Let $U^- = \bbZ_p[\Delta][-1] \oplus \bbZ_p[\Delta][-2]$, considered
as a perfect complex of $\La$-modules, or as a complex of
$\La_n$-modules.  One constructs a map $i^-_n \cn \bfR\Ga_n \to U^-$
via the local valuation maps in degree one and the local invariant
maps in degree two, and obtains a map $i^-_\Iw \cn \bfR\Ga_\Iw \to
U^-$ from the $i^-_n$ by taking the inverse limit on $n$.  By taking
mapping fibers of $i^-_n$ and $i^-_\Iw$, one obtains complexes
$\bfR\wt\Ga_{f,n}$ of $\La_n$-modules and a perfect complex
$\bfR\wt\Ga_{f,\Iw}$ of $\La$-modules sitting in exact triangles
\[
\bfR\wt\Ga_{f,n}
\to
\bfR\Ga_n
\xrightarrow{i^-_n}
U^-
\to
\bfR\wt\Ga_{f,n}[1]
\]
and
\[
\bfR\wt\Ga_{f,\Iw}
\to
\bfR\Ga_\Iw
\xrightarrow{i^-_\Iw}
U^-
\to
\bfR\wt\Ga_{f,\Iw}[1].
\]
Writing $\wt H^i_{f,?} = H^i(\bfR\wt\Ga_{f,?})$ for $? \in
\{n,\Iw\}$, one has the computations
\[
\wt H^i_{f,n} = \begin{cases}
0 & i \neq 1,2,3 \\
\scrE(K_n) & i=1 \\
A(K_n) & i=2 \\
\bbZ_p & i=3,
\end{cases}
\qquad \text{and} \qquad
\wt H^i_{f,\Iw} = \begin{cases}
0 & i \neq 1,2,3 \\
\scrE(K_\infty) & i=1 \\
A(K_\infty) & i=2 \\
\bbZ_p & i=3.
\end{cases}
\]

By control for Galois cohomology, the natural map $\bfR\Ga_\Iw
\Lotimes_\La \La_n \to \bfR\Ga_n$ is an isomorphism, compatible with
varying $n$.  Since $U^- = (U^-)^\Ga$, one has the computation $U^-
\Lotimes_\La \La_n \cong U^-[1] \oplus U^-$.  It follows from the
definition of $i^-_\Iw$ as an inverse limit that $i^-_\Iw \otimes_\La
\La_n = i^-_n$, so that $i^-_\Iw \Lotimes_\La \La_n = (\beta_{U^-}\inv
\circ i^-_n[1] \circ \beta_{\bfR\Ga_n}, i^-_n)$.  Thus we have a
commutative diagram
\[\begin{array}{ccccccc}
\bfR\wt\Ga_{f,\Iw} \Lotimes_\La \La_n
& \to &
\bfR\Ga_n
& \xrightarrow{i^-_\Iw \Lotimes_\La \La_n} &
U^-[1] \oplus U^-
& \to &
\bfR\wt\Ga_{f,\Iw} \Lotimes_\La \La_n[1] \\
& & =\downarrow\phantom= & & \pr_2\downarrow\phantom{\pr_2} \\
\bfR\wt\Ga_{f,n}
& \to &
\bfR\Ga_n
& \xrightarrow{i^-_n} &
U^-
& \to &
\bfR\wt\Ga_{f,n}[1],
\end{array}\]
which we complete to a morphism of exact triangles via a morphism
$\BC_n \cn \bfR\wt\Ga_{f,\Iw} \Lotimes_\La \La_n \to
\bfR\wt\Ga_{f,n}$.  Taking mapping fibers of the resulting morphism of
triangles gives an exact triangle
\[
\Fib(\BC_n) \to 0 \to U^-[1] \to \Fib(\BC_n)[1],
\]
hence an isomorphism $\Fib(\BC_n) \cong U^-$ and an exact triangle
\begin{equation}\label{E:BC-cone}
U^-
\xrightarrow{j_n}
\bfR\wt\Ga_{f,\Iw} \Lotimes_\La \La_n
\xrightarrow{\BC_n}
\bfR\wt\Ga_{f,n}
\xrightarrow{k_n}
U^-[1].
\end{equation}
It is easy to compute that $j_n$ is the composite of the inclusion
$U^- \hookrightarrow U^- \oplus U^-[-1]$ and the shifted connecting
homomorphism $U^- \oplus U^-[-1] \to \bfR\wt\Ga_{f,\Iw} \Lotimes_\La
\La_n$.  The construction of the snake lemma shows that $k_n$ is the
composite
\[
\bfR\wt\Ga_{f,n}
\to
\bfR\Ga_n
\xrightarrow{i^-_\Iw \Lotimes_\La \La_n}
U^-[1] \oplus U^-
\xrightarrow{\pr_1}
U^-[1],
\]
or in other words the composite of $\bfR\wt\Ga_{f,n} \to \bfR\Ga_n$
with $\beta_{U^-}\inv \circ i^-_n[1] \circ \beta_{\bfR\Ga_n}$.  Of
course, the source or target of $H^i(k_n) \cn \wt H^i_{f,n} \to
H^{i+1}U^-$ is zero if $i \neq 1$, and if $i=1$ this computation
simplifies to
\[
\scrE(K_n)
\to
H^1_n
\xrightarrow\beta
H^2_n
\xrightarrow{\invt}
\bbZ_p[\Delta].
\]
The kernel of $\beta$ contains the universal norms in $\scrE(K_n)$ for
$K_\infty/K_n$, and in particular $\scrC(K_n)$ (see
\cite[Proposition~II.2.5]{D} for the norm relations), which itself is
of finite index in $\scrE(K_n)$.  Since $\bbZ_p[\Delta]$ is torsion
free, it follows that $H^1(k_n) = 0$, too.  Since $H^*(k_n)=0$, the
long exact sequence associated to the triangle \eqref{E:BC-cone}
breaks up into the short exact rows in the following diagrams, and
\eqref{E:generic-base-change} gives the short exact columns:
\begin{equation}\label{E:crosses}
\begin{array}{r@{\ }c@{\ }l}
& \scrE(K_\infty)_{\Ga^{p^n}} \\ & \downarrow \\
\bbZ_p[\Delta] \to & H^1(\bfR\wt\Ga_{f,\Iw} \Lotimes_\La \La_n)
  & \to \scrE(K_n), \\
& \downarrow \\ & A(K_\infty)^{\Ga^{p^n}}
\end{array}
\quad
\begin{array}{r@{\ }c@{\ }l}
& A(K_\infty)_{\Ga^{p^n}} \\ & \downarrow \\
\bbZ_p[\Delta] \to & H^2(\bfR\wt\Ga_{f,\Iw} \Lotimes_\La \La_n)
  & \to A(K_n). \\
& \downarrow \\ & \bbZ_p
\end{array}
\end{equation}
(The triangle \eqref{E:BC-cone} also gives the computation
$H^3(\bfR\wt\Ga_{f,\Iw} \Lotimes_\La \La_n) \cong \bbZ_p$.)  The
composite arrows from the top to the right points of the two diagrams
give the respective control maps $\scrE(K_\infty)_{\Ga^{p^n}} \to
\scrE(K_n)$ and $A(K_\infty)_{\Ga^{p^n}} \to A(K_n)$.

It is crucial to compute the transition morphisms from the diagrams
\eqref{E:crosses} associated to $n+1$ to those associated to $n$.
Explicitly, the transition maps on the upper (resp.\ lower, right)
points are the natural projections (resp.\ multiplication by
$1+\ga^{p^n}+\cdots+\ga^{(p-1)p^n}$, the norm maps), and the maps on
the left points are \emph{multiplication by $p$} because the term
$U^-[1]$ in the sequence \eqref{E:BC-cone} is identified with first
summand of $U^- \Lotimes_\La \La_n \cong U^-[1] \oplus U^-$.

We consider the second diagram in \eqref{E:crosses}.  The computation
$\rank_{\bbZ_p} A(K_\infty)_{\Ga^{p^n}} = \de-1$ is immediate.  A
diagram chase identifies the composite map $\bbZ_p[\Delta] \to \bbZ_p$
as the sum map.  Since $\bbZ_p$ is uniquely a $\La$-direct summand of
$\bbZ_p[\Delta]$, we may canonically refine the diagram to the short
exact sequence
\begin{equation}\label{E:SES-A}
0 \to \bbZ_p[\Delta]^!
\to A(K_\infty)_{\Ga^{p^n}} \to
A(K_n) \to 0,
\end{equation}
and in particular there is an injection
$A(K_\infty)_{\Ga^{p^n}}[p^\infty] \hookrightarrow A(K_n)$ of finite
abelian groups.  Applying the snake lemma to the commutative diagram
\[\begin{array}{rcccccl}
0 \to &
\bbZ_p[\Delta]^!
& \to &
\displaystyle
\frac{A(K_\infty)_{\Ga^{p^{n+1}}}}{A(K_\infty)_{\Ga^{p^{n+1}}}[p^\infty]}
& \to &
\displaystyle
\frac{A(K_{n+1})}{A(K_\infty)_{\Ga^{p^{n+1}}}[p^\infty]}
& \to 0 \\
& p \downarrow \phantom{p} & & \downarrow & & \downarrow \\
0 \to &
\bbZ_p[\Delta]^!
& \to &
\displaystyle
\frac{A(K_\infty)_{\Ga^{p^n}}}{A(K_\infty)_{\Ga^{p^n}}[p^\infty]}
& \to &
\displaystyle
\frac{A(K_n)}{A(K_\infty)_{\Ga^{p^n}}[p^\infty]}
& \to 0,
\end{array}\]
and examining the final column, we get the exact sequence
\[
0 \to \bbZ_p[\Delta]^!/p
\to
\frac{A(K_{n+1})}{A(K_\infty)_{\Ga^{p^{n+1}}}[p^\infty]}
\to
\frac{A(K_n)}{A(K_\infty)_{\Ga^{p^n}}[p^\infty]} \to 0.
\]
This implies that
\[
\frac{\#A(K_n)}{\#A(K_\infty)_{\Ga^{p^n}}[p^\infty]}
=
p^{(\de-1)n} \frac{\#A(K_0)}{\#A(K_\infty)_{\Ga^{p^0}}[p^\infty]},
\]
so that
\begin{gather*}
\mu(A(K_\infty))
= \mu(\{\#A(K_\infty)_{\Ga^{p^n}}[p^\infty]\})
= \mu(\{\#A(K_n)\}), \\
\la(A(K_\infty))
= \de-1 + \la(\{\#A(K_\infty)_{\Ga^{p^n}}[p^\infty]\})
= \la(\#\{A(K_n)\}).
\end{gather*}

We now consider the first diagram in \eqref{E:crosses}.  Since
$\scrE(K_n)$ is a free $\bbZ_p$-module, we may choose a splitting
$H^1(\bfR\wt\Ga_{f,\Iw} \Lotimes_\La \La_n) \cong \bbZ_p[\Delta]
\oplus \scrE(K_n)$.  It follows from the norm relations on elliptic
units that the map $\scrC(K_\infty)_{\Ga^{p^n}} \to \scrC(K_n)$ is
surjective; combining this fact with the proof of
\cite[Theorem~7.7]{R} shows that ther kernel of this map is
$(\scrC(K_\infty)_{\Ga^{p^n}})^\Ga$, is $\scrG$-isomorphic to $\bbZp$,
and is a $\bbZ_p$-direct summand of $\scrC(K_\infty)_{\Ga^{p^n}}$.
Moreover, this map followed by the inclusion $\scrC(K_n) \subseteq
\scrE(K_n)$ is equal to the composite
\[
\scrC(K_\infty)_{\Ga^{p^n}} \to \scrE(K_\infty)_{\Ga^{p^n}} \to
\scrE(K_n),
\]
which shows that the subset of $\scrC(K_\infty)_{\Ga_{p^n}}$ mapping
into the summand $\bbZ_p[\Delta]$ is again
$(\scrC(K_\infty)_{\Ga_{p^n}})^\Ga$.  Writing $\bbZ_p[\Delta] =
\bbZ_p[\Delta]^{\bf1} \oplus \bbZ_p[\Delta]^!$, the image $I_n$ of
$(\scrC(K_\infty)_{\Ga_{p^n}})^\Ga \to \bbZ_p[\Delta]$ is equal to
either $0$ or $p^{e_n}\bbZ_p[\Delta]^{\bf1}$ with $e_n \geq 0$.  If
$I_n = 0$ we set $e_n = 0$, so that in all cases we have
$(\bbZ_p[\Delta]/I_n)[p^\infty] \cong \bbZ/p^{e_n}$.  Write $\ep_n = 1
- \rank_{\bbZ_p} I_n$.  Again considering the proof of
\cite[Theorem~7.7]{R} shows that, in the commutative diagram
\[\begin{array}{rcccccl}
0 \to &
\bbZ_p
& \to &
\scrC(K_\infty)_{\Ga^{p^{n+1}}}
& \to &
\scrC(K_{n+1})
& \to 0 \\
& f \downarrow \phantom{f} & & \downarrow & & \downarrow \\
0 \to &
\bbZ_p
& \to &
\scrC(K_\infty)_{\Ga^{p^n}}
& \to &
\scrC(K_n)
& \to 0,
\end{array}\]
the map $f$ is multiplication by $p$ (up to a unit).  Let $v_n$ be a
basis vector for $I_n$ if $\ep_n=0$, and $v_n=0$ if $\ep_n=1$.  The
commutativity of the square
\[\begin{array}{cccc}
\bbZ_p & \xrightarrow{\cdot v_{n+1}} & I_{n+1} \subseteq & \bbZ_p[\Delta] \\
f \downarrow \phantom{f} & & & \phantom{p} \downarrow p \\
\bbZ_p & \xrightarrow{\cdot v_n} & I_n \subseteq & \bbZ_p[\Delta] \\
\end{array}\]
implies that $v_n=0$ if and only if $v_{n+1}=0$, so that
$\ep_n=\ep_{n+1}$ is independent of $n$; denote it henceforth by
$\ep$.  When $\ep=0$, it is also easy to deduce from this
commutativity that $e_n=e_{n+1}$ is independent of $n$; denote it
henceforth by $e$.

The definition of $I_n$ allows us to modify the first diagram in
\eqref{E:crosses} to a short exact sequence
\begin{equation}\label{E:SES-EC}
0 \to (\scrE(K_\infty)/\scrC(K_\infty))_{\Ga^{p^n}}
\to
\bbZ_p[\Delta]/I_n \oplus \scrE(K_n)/\scrC(K_n)
\to
A(K_\infty)^{\Ga^{p^n}} \to 0.
\end{equation}
One has $\rank_{\bbZ_p} A(K_\infty)^{\Ga^{p^n}} = \rank_{\bbZ_p}
A(K_\infty)_{\Ga^{p^n}} = \de-1$, and combining this with the above
sequence gives $\rank_{\bbZ_p}
(\scrE(K_\infty)/\scrC(K_\infty))_{\Ga^{p^n}} = \ep$.  The above
sequence also gives an exact sequence of finite abelian groups
\[
0
\to (\scrE(K_\infty)/\scrC(K_\infty))_{\Ga^{p^n}}[p^\infty]
\to \bbZ/p^e \oplus \scrE(K_n)/\scrC(K_n)
\to A(K_\infty)^{\Ga^{p^n}}[p^\infty],
\]
where $\#A(K_\infty)^{\Ga^{p^n}}[p^\infty]$ is bounded independently
of $n$.  It follows that
\begin{gather*}
\mu(\scrE(K_\infty)/\scrC(K_\infty))
  = \mu(\{\#\scrE(K_n)/\scrC(K_n)\}), \\
\la(\scrE(K_\infty)/\scrC(K_\infty))
  = \ep + \la(\{\#\scrE(K_n)/\scrC(K_n)\}).
\end{gather*}

The analytic class number formula gives
\[
\#A(K_n) = \#\scrE(K_n)/\scrC(K_n),
\]
and the computations $\rank_{\bbZ_p} A(K_\infty)_\scrG = 0$ and
$\rank_{\bbZ_p} (\scrE(K_\infty)/\scrC(K_\infty))_\scrG = \ep$ follow
from \eqref{E:SES-A} and \eqref{E:SES-EC}.  This establishes
\eqref{E:subtheorem}, and completes the proof of the theorem.

\end{document}